\documentclass{amsart}

\usepackage{amsmath, amssymb}

\newtheorem{lettertheorem}{Theorem}

\newtheorem{theorem}{Theorem}[section]

\newtheorem{proposition}[theorem]{Proposition}
\newtheorem{example}[theorem]{Example}
\newtheorem{remark}[theorem]{Remark}

\begin{document}

\title{Meromorphic solutions of algebraic difference equations}

\thanks{The first author would like to thank the support of the discretionary budget (2016) of the President of the Open University of Japan.
The second author would like to thank the partial support by the Academy of Finland grants (\#286877) and (\#268009).}





\author{Katsuya Ishizaki \and Risto Korhonen}




\maketitle

\begin{abstract}
It is shown that the difference equation
\begin{equation}\label{abseq}
(\Delta f(z))^2=A(z)(f(z)f(z+1)-B(z)),
\end{equation}
where $A(z)$ and $B(z)$ are meromorphic functions, possesses a continuous limit to the differential equation
    \begin{equation}\label{abseq2}
    (w')^2=A(z)(w^2-1),
    \end{equation}
which extends to solutions in certain cases. In addition, if \eqref{abseq} possesses two distinct transcendental meromorphic solutions, it is shown that these solutions satisfy an algebraic relation, and that their growth behaviors are almost same in the sense of Nevanlinna under some conditions. Examples are given to discuss the sharpness of the results obtained. These properties are counterparts of the corresponding results on the algebraic differential equation \eqref{abseq2}.
\end{abstract}

\numberwithin{equation}{section}

\section{Introduction}
According to the classical Malmquist's theorem \cite{malmquist:13}, if the differential equation
    \begin{equation}\label{malmeq}
    w'=R(z,w),
    \end{equation}
where $R(z,w)$ is rational in both arguments, has a transcendental meromorphic solution, then \eqref{malmeq} reduces into a Riccati equation. Steinmetz \cite{steinmetz:78}, Bank and Kaufman \cite{bankk:80} have generalized Malmquist's result by showing that the equation
    \begin{equation}\label{malmeq2}
    (w')^n=R(z,w),
    \end{equation}
where again $R(z,w)$ is rational in both arguments, reduces into one in a list of six simple differential equations, after a suitable M\"obius transformation. One of these six equations is
    \begin{equation}\label{diffeq}
    (w')^2 = a(z)(w-b(z))^2(w-\tau_1)(w-\tau_2),
    \end{equation}
where $\tau_1,\tau_2$ are constants and $a(z)$ and $b(z)$ are rational functions. Equation \eqref{diffeq} can be transformed into
    \begin{equation}
    (w')^2=A(z)(w^2-1)\label{1.5}
    \end{equation}
by a simple linear transformation, provided that the coefficient $b(z)$ is a constant.

In this paper, we are concerned with meromorphic solutions of a difference equation of the form
\begin{equation}
(\Delta f(z))^2=A(z)(f(z)f(z+1)-B(z)),\label{1.1}
\end{equation}
where $A(z)$ and $B(z)$ are meromorphic functions, and $\Delta$ is the difference operator $\Delta f(z)=f(z+1)-f(z)$. We will show in Section~\ref{limitSec} below that \eqref{1.1} possesses a continuous limit to \eqref{1.5}, which, in some cases, extends to an explicit limit between the solutions of \eqref{1.1} and \eqref{1.5}, as well. In addition, we consider the order of growth and the value distribution of a meromorphic solution of \eqref{1.1}, see e.g.,~\cite{goldbergo:08},~\cite{hayman:64},~\cite{nevanlinna:70} for a description of Nevanlinna's value distribution theory. One of the questions to be
considered is whether
\begin{equation}
T(r,f_1)=T(r,f_2)(1+o(1)), \quad \text{as $r\to \infty$, $r\not\in E$}\label{1.4}
\end{equation}
holds for distinct meromorphic solutions $f_1(z)$ and $f_2(z)$ of \eqref{1.1}, where
$E$ is an exceptional set with finite logarithmic measure.
In \cite{ishizakit:07}, an algebraic relation for distinct meromorphic solutions to differential equation \eqref{1.5} is obtained, which yields that \eqref{1.4} is satisfied for any distinct transcendental
meromorphic solutions to \eqref{1.5} in place of $f_j$, $j=1, 2$. We discuss this result in more detail in Section~\ref{periodicSec} below.

Nevanlinna growth considerations of an equation related to \eqref{1.1} were considered by Liu in \cite{liu:14} in the following sense. Equation~\eqref{1.1} can be also written in the form
\begin{equation}
(\Delta f(z))^2-A(z)f(z)\Delta f(z)-A(z)f(z)^2+A(z)B(z)=0\label{1.2},
\end{equation}
which belongs to a class of algebraic difference equations of first order,
\begin{equation}
(\Delta g(z))^2+P(z,g(z))\Delta g(z)+Q(z,g(z))=0,\label{1.3}
\end{equation}
where $P(z,g)$ and $Q(z,g)$ are polynomials in $g$ with meromorphic coefficients.
Assume that all coefficients in \eqref{1.3} are rational functions and \eqref{1.3} possesses
a transcendental meromorphic solution $g(z)$ of finite order. Suppose in addition that $\deg_gP(z,g)\leq1$ and $\deg_gQ(z,g)=4$ (or $3$). Then,
by means of the difference analogues of the lemma on the logarithmic derivatives,
$g(z)$ has infinitely many poles, see e.g.,~\cite{halburdk:06JMAA},~\cite{chiangf:08},~\cite{lainey:07}.
On the other hand, by the power test for poles of a meromorphic solution $g(z)$ to \eqref{1.3},
$g(z)$ has only finitely many poles. This contradiction implies that \eqref{1.3} has no meromorphic solution of finite order, that is a generalization of Yanagihara's theorem~\cite{yanagihara:80} in some sense. By following this argument, and by using an extension of the the difference analogues of the lemma on the logarithmic derivatives obtained in \cite{halburdkt:14TAMS}, it follows that \eqref{1.3} does not have any transcendental meromorphic solutions of hyper-order $<1$.

This paper consists of five sections.
In Section~\ref{examplesSec}, we give several examples of \eqref{1.1}.
Section~\ref{limitSec} is devoted to showing that \eqref{1.5} can be obtained from \eqref{1.1} by a continuous limit, which also extends to solutions under some conditions.
We consider the case when $A(z)$ and $B(z)$ are periodic functions of period 1 in Section~\ref{periodicSec}.
An algebraic relation among solutions to \eqref{1.1} is discussed.
With some conditions in Section~\ref{VG} it is shown that \eqref{1.4} holds for transcendental meromorphic solutions to \eqref{1.1}.


\begin{remark}\label{rem1.1}
We consider the case $A(z)$ and $B(z)$ are constants in \eqref{1.1}.
For a solution $f(z)$ to \eqref{1.1}, we define $\hat{f}(z)=f(\kappa(z)+z)$, where $\kappa(z)$ is a periodic function of period 1.
Substituting $\kappa(z)+z$ in \eqref{1.1} in place of $z$, we see that $\hat{f}(z)$ is also
a solution of \eqref{1.1}, since $\hat{f}(z+1)=f(\kappa(z+1)+z+1)=f(\kappa(z)+z+1)$.
This means that the order of $\hat{f}(z)$ may be bigger than the order of $f(z)$,
and hence \eqref{1.4} is not always valid in general.
\end{remark}

\section{Examples}\label{examplesSec}

We have collected to this section a number of examples to which we keep referring to throughout the remainder of this paper.

\begin{example}\label{ex2.1}
Let $a\ne0$ be a constant, and set $f_1(z)=\sin az$.
We have $\Delta f_1(z)=2\sin \frac{a}{2}\cos a(z+\frac{1}{2})$, see e.g., \cite[Theorem~2.2]{kelleyp:01}.
This gives
\begin{equation*}
(\Delta f(z))^2=-\left(2\sin \frac{a}{2}\right)^2\left(\sin^2a\left(z+\frac{1}{2}\right)-1\right).
\end{equation*}
For the sake of simplicity we write $\alpha=\cos\frac{a}{2}$ and $\beta=\sin\frac{a}{2}$.
Then
\begin{align*}
\sin^2 a\left(z+\frac{1}{2}\right)&=\alpha^2\sin^2az+2\alpha\beta\sin az\cos az+\beta^2 \cos^2az\\
&=\sin az\left(\sqrt{(\alpha^2-\beta^2)^2+(2\alpha\beta)^2}\sin (az+\theta)\right)+\beta^2\\
&=\sin az \sin a\left(z+\frac{\theta}{a}\right)+\beta^2,
\end{align*}
with $\tan\theta=2\alpha\beta/(\alpha^2-\beta^2)=\tan a$. Hence $f_1(z)$ satisfies
\begin{equation}
(\Delta f(z))^2=-4\sin^2\frac{a}{2}\left(f(z)f(z+1)-\cos^2\frac{a}{2}\right).\label{2.1}
\end{equation}
By similar computations as above, we see that $f_2(z)=\cos az$ also satisfies the same difference equation \eqref{2.1}, which corresponds to \eqref{1.1} with $A(z)=-4\sin^2\frac{a}{2}$ and
$B(z)=\cos^2\frac{a}{2}$.
On the other hand, by means of Remark~\ref{rem1.1}, \eqref{1.4} does not always hold in general.
\end{example}

\begin{example}\label{ex2.2} Let $b$ be a non-zero complex number. The function
\begin{equation}
f_b(z)=\frac{1+be^{\pi i z}-e^{2\pi i z}}{e^{2 \pi i z}-1}\label{2.2}
\end{equation}
satisfies a difference equation ($A(z)=-4$ and $B(z)=1$ in \eqref{1.1})
\begin{equation}
(\Delta f(z))^2=-4(f(z)f(z+1)-1).\label{2.3}
\end{equation}
The functions $-f_b(z)$, $f_{b'}(z)$, $b'\ne b$ also satisfy \eqref{1.4}.
They have obviously the same order of growth.
However, there exists higher order solutions to \eqref{2.3}, which implies that \eqref{1.4} does not hold in general as we mentioned in Remark~\ref{rem1.1}.
Clearly, $f_b(z)$ has infinitely many poles, which are on real axis at $z=0$, $\pm 1$, $\pm 2$, \dots. Thus we have
$N(r,f_b)\sim r$, and hence $\infty$ is not a deficiency for $f_b(z)$.
By Remark~\ref{rem1.1}, we may construct meromorphic solutions to
\eqref{2.3} having infinitely many poles of higher order.
\end{example}

\begin{example}\label{ex2.3} Similar idea to Example~\ref{ex2.2} yields another example.
Let $\beta(z)\not\equiv0$ be arbitrary periodic function of period 1. The function
\begin{equation}
f_\beta(z)=\frac{1-\beta(z) e^{\pi i z}+e^{2\pi i z}}{e^{2 \pi i z}-1}\label{2.4}
\end{equation}
satisfies a difference equation ($A(z)=-4\beta(z)^2/(\beta(z)^2-4)$, and $B(z)=1$ in \eqref{1.1})
\begin{equation}
(\Delta f(z))^2=-\frac{4\beta(z)^2}{\beta(z)^2-4}(f(z)f(z+1)-1).\label{2.5}
\end{equation}
\end{example}

\begin{example}\label{ex2.4}  Let $Q(z)$ be a periodic function of period 1. The function
\begin{equation}
f(z)=\frac{z^2+Q(z)^2}{2Q(z)z}\label{2.6}
\end{equation}
satisfies a difference equation ($A(z)=1/(z(z+1))$, and $B(z)=(1+2z)^2/(4z(z+1))$ in \eqref{1.1})
\begin{equation}
(\Delta f(z))^2=\frac{1}{z(z+1)}\left(f(z)f(z+1)-\frac{(1+2z)^2}{4z(z+1)}\right).\label{2.7}
\end{equation}
\end{example}

\section{Continuous limit}\label{limitSec}
In this section we apply continuous limits to give connections between solutions of certain classes of difference equations and solutions of the corresponding differential equations.
Shimomura~\cite{shimomura:12} applied the continuous limit from the discrete second Painlev\'e equation to
the second differential Painlev\'e equation in order to
observe asymptotic expansions of solutions.
Continuous limit has  been contributed mostly to Painlev\'e analysis, e.g.,~\cite[\S 50]{gromakls:02},~\cite{quispelcs:92},~\cite{ramanigh:91}. In~\cite{ishizaki:17}, the difference Riccati equation was discussed.

By a continuous limit we mean broadly the following. Let $k$ be a positive integer. We consider a difference equation
\begin{equation}
\Omega_0(z,f(z), f(z+1), \dots, f(z+k))=0.\label{3.1}
\end{equation}
Let $\varepsilon$ be a complex number. We set a pair of relations
\begin{equation}
\mu(z,t,\varepsilon)=0 \quad\text{ and}\quad\nu(f(z), w(t,\varepsilon), \varepsilon)=0.\label{3.2}
\end{equation}
According to \eqref{3.2}, we transform \eqref{3.1} to a certain difference equation
\begin{equation}
\Omega_1(t,w(t,\varepsilon), w(t+\varepsilon,\varepsilon), \dots, w(t+k\varepsilon,\varepsilon))=0\label{3.3}
\end{equation}
with some conditions on coefficients of $\Omega_1$.
Letting $\varepsilon\to0$, we derive a differential equation
\begin{equation}
\Omega_2(t,w(t,0),w'(t,0),\dots, w^{(k)}(t,0))=0.\label{3.4}
\end{equation}
Depending on the conditions on coefficients, the arguments above give wide scope of ideas.
On the other hand, it would remain generally within the framework of the formal discussion.

We discuss next a continuous limit from \eqref{1.1} to \eqref{1.5}. To do this, we set
\begin{equation}
t=\varepsilon z\quad \text{and}\quad f(z)=w(t,\varepsilon),\label{3.5}
\end{equation}
in \eqref{1.1} and give $\varepsilon^2 \tilde A(t,\varepsilon)$ and $\tilde{B}(t,\varepsilon)$
in place of $A(z)$ and $B(z)$, respectively.
Since $f(z+1)=w(\varepsilon(z+1),\varepsilon)=w(\varepsilon z+\varepsilon,\varepsilon)= w(t+\varepsilon,\varepsilon)$, we have
\begin{equation}
(w(t+\varepsilon,\varepsilon)-w(t,\varepsilon))^2=\varepsilon^2\tilde{A}(t,\varepsilon)(w(t,\varepsilon)w(t+\varepsilon,\varepsilon)-\tilde{B}(t,\varepsilon)).\label{3.6}
\end{equation}
Assume that
\begin{equation}
\lim_{\varepsilon\to0}\tilde{A}(t,\varepsilon)=\tilde{A}(t,0)\quad\text{and}\quad \lim_{\varepsilon\to0}\tilde{B}(t,\varepsilon)=\tilde{B}(t,0).\label{3.7}
\end{equation}
Letting $\varepsilon\to0$, we see that $\displaystyle w(t,0)=\lim_{\varepsilon\to0}w(t,\varepsilon)$, if exists, satisfies the differential equation
\begin{equation}
w'(t)^2=\tilde A(t)(w(t)^2-\tilde B(t)),\label{3.8}
 \end{equation}
with $\tilde A(t)=\tilde A(t,0)$ and $\tilde B(t)=\tilde B(t,0)$.
If we set
\begin{equation}
A(z)=\varepsilon^2 \tilde A(t,\varepsilon)\quad \text{and}\quad B(z)=\tilde{B}(t,\varepsilon),\label{3.9}
\end{equation}
then this builds a \textit{direct connection}. When we do not assume \eqref{3.9}, we call the derivation
\eqref{3.6} an \textit{indirect connection}. We compare these connections in Examples~\ref{ex3.1} and~\ref{ex3.2} below.

\begin{example}\label{ex3.1}
We observe \eqref{2.7} in Example~\ref{2.4} in view of the continuous limit assuming the condition \eqref{3.9}, i.e., a direct connection.
According to \eqref{3.5} and \eqref{3.9}, we derive a difference equation of the form
\eqref{3.6} from \eqref{2.7}. We compute
$$
\tilde A(t,\varepsilon)=\frac{1}{\varepsilon^2}\frac{1}{z(z+1)}=\frac{1}{\varepsilon^2}\frac{1}{\frac{t}{\varepsilon}(\frac{t}{\varepsilon}+1)}=\frac{1}{t(t+\varepsilon)}
$$
and
$$
\tilde B(t,\varepsilon)=\frac{(1+2z)^2}{4z(z+1)}=\frac{(1+\frac{2t}{\varepsilon})^2}{4\frac{t}{\varepsilon}(\frac{t}{\varepsilon}+1)}=\frac{(2t+\varepsilon)^2}{4t(t+\varepsilon)},
$$
which implies that $\tilde{A}(t,0)=1/t^2$ and $\tilde{B}(t,0)=1$.
The corresponding equation given by \eqref{3.8}
\begin{equation}
(w')^2=\frac{1}{t^2}(w^2-1)\label{3.10}
\end{equation}
possesses a solution $w_C(t)=(C^2+t^2)/(2Ct)$, where $C$ is an arbitrary constant.
By \eqref{2.6} and \eqref{3.5}, we have
\begin{equation*}
w(t,\varepsilon)=\frac{(\frac{t}{\varepsilon})^2+Q(\frac{t}{\varepsilon})^2}{2Q(\frac{t}{\varepsilon})\frac{t}{\varepsilon}}=\frac{t^2+(\varepsilon Q(\frac{t}{\varepsilon}))^2}{2t(\varepsilon Q(\frac{t}{\varepsilon}))},
\end{equation*}
which corresponds to $w_C(z)$ if $\displaystyle \lim_{n\to0} \varepsilon_n Q(\tfrac{t}{\varepsilon_n})=C$ would hold for some $\{\varepsilon_n\}$, $\varepsilon_n\to0$.
In fact, we assume that $Q(z)$ is a non-constant meromorphic periodic function of period 1.
Since $Q(z)$ is transcendental, $Q(tz)-Cz$ has infinitely many zeros $\{z_n\}$,
$z_n\to\infty$ for fixed $t$ and $C$
with at most two exceptions. Then we define $\varepsilon_n=1/z_n$, which implies the assertion.
\end{example}

\begin{example}\label{ex3.2}
Consider \eqref{2.3} in Example~\ref{2.2}.
Here we do not assume the condition \eqref{3.9}, i.e., we are aiming for an indirect connection.
By \eqref{3.5}, the difference equation \eqref{2.3} is transformed
\begin{equation}
(w(t+\varepsilon,\varepsilon)-w(t,\varepsilon))^2=-4(w(t,\varepsilon)w(t+\varepsilon,\varepsilon)-1)\label{3.11}
\end{equation}
We set $\varepsilon^2 \tilde A(t,\varepsilon)$ in place of $-4$ in \eqref{3.11}, where
$\tilde A(t,\varepsilon)=-4\sin^2\varepsilon/\varepsilon^2$, and $\tilde B(t,\varepsilon)=\sin^2\varepsilon-1$ in place of $-1$ in \eqref{3.11}. Then we obtain a difference equation
corresponding to \eqref{3.3}
\begin{equation}
(w(t+\varepsilon,\varepsilon)-w(t,\varepsilon))^2=-4\sin^2\varepsilon(w(t,\varepsilon)w(t+\varepsilon,\varepsilon)+\sin^2\varepsilon-1).\label{3.12}
\end{equation}
The function $w(t+\varepsilon,\varepsilon)=\sin(2t+\varphi(\varepsilon))$ satisfies \eqref{3.11}, which is confirmed by the formula $\sin^2(2h+k)=\sin(2(h+k))\sin2h+\sin^2k$.
Assume that $\displaystyle \lim_{\varepsilon\to0}\varphi(\varepsilon)=0$.
Then letting $\varepsilon\to0$ in \eqref{3.12}, we see that $w(t,0)=\sin2t$ satisfies a differential equation corresponding to \eqref{3.8} with
 $\tilde{A}(t,0)=-4$ and $\tilde{B}(t,0)=1$, i.e.,
\begin{equation}
(w')^2=-4(w^2-1).\label{3.13}
\end{equation}
Applying an indirect connection,
we can start to discuss a difference equation \eqref{3.3} which
has close properties to a differential equation \eqref{3.4}.
This argument could yield various discoveries. On the other hand, indirect connections
might lose some properties of the original difference equation \eqref{3.1}.
For example, a transcendental meromorphic solution $f_b(z)$ to \eqref{2.3} given by \eqref{2.2} possesses infinitely many poles.
However, any transcendental meromorphic solution of \eqref{3.13} has no poles.
\end{example}

Most of the continuous limits obtained so far in the literature, see, e.g., \cite{grammaticosnr:99,halburdk:07JPA} and the reference therein, are indirect.

\section{Periodic case}\label{periodicSec}
In this section, we consider the case when $A(z)$ and $B(z)$ are periodic functions of period 1.
We have
\begin{proposition}\label{prop4.1}
Suppose that $A(z)$ and $B(z)$ are periodic functions of period 1 in \eqref{1.1},
and suppose that \eqref{1.1} possesses a meromorphic solution $f(z)$.
Then either $f(z)$ is a periodic function of period 2,
or $f(z)$ satisfies a linear difference equation of second order
\begin{equation}
\Delta^2 f(z)-A(z)\Delta f(z)-A(z)f(z)=0.\label{4.1}
\end{equation}
\end{proposition}


\begin{proof}
From \eqref{1.1},
\begin{align*}
(\Delta f(z+1))^2&=A(z+1)(f(z+1)f(z+2)-B(z+1))\\
&=A(z)(f(z+1)f(z+2)-B(z)).
\end{align*}
Eliminating $A(z)$ from this equation and \eqref{1.1}, we obtain
\begin{multline*}
(f(z+2)-f(z))\\
\times \left(f(z+1)^3-(2f(z+1)-f(z))B(z)+(B(z)-f(z)f(z+1))f(z+2)\right)=0,
\end{multline*}
which implies that $f(z)$ is a periodic function of period 2, or
\begin{equation}
f(z+2)=\frac{f(z+1)^3-(2f(z+1)-f(z))B(z)}{f(z)f(z+1)-B(z)}.\label{4.2}
\end{equation}
Using \eqref{4.2} and \eqref{1.1}, we have
\begin{align*}
\Delta^2 f(z)&=f(z+2)-2f(z+1)+f(z)\\
&=\frac{f(z+1)^3-(2f(z+1)-f(z))B(z)}{f(z)f(z+1)-B(z)}-2f(z+1)+f(z)\\
&=\frac{(\Delta f(z))^2f(z+1)}{f(z)f(z+1)-B(z)}=A(z)f(z+1),
\end{align*}
which concludes \eqref{4.1}.
\end{proof}


\begin{remark}\label{rem4.1} Two functions $f_1(z)$ and $f_2(z)$
in Example~\ref{ex2.1} are not always periodic
functions of period 2. They satisfy \eqref{4.1} with $A(z)=-4\sin^2\frac{a}{2}$.
Functions $f_b(z)$ in Example~\ref{ex2.2}, and $f_\beta(z)$ in Example~\ref{ex2.3} are periodic functions of period 2.
\end{remark}

The following result on the algebraic dependence of the solutions of the differential equation \eqref{1.5} was obtained in \cite[Theorem~2.1]{ishizakit:07}.



%

\begin{lettertheorem}[\cite{ishizakit:07}]
Assume that $A(z)$ is a rational function in \eqref{1.5}. Suppose that \eqref{1.5} possesses distinct
transcendental meromorphic solutions $w_1$ and $w_2$. Then there exists a constant $c$ such that
\begin{equation}
w_1^2+2cw_1w_2+w_2^2=1-c^2.\label{4.3}
\end{equation}
\end{lettertheorem}

The algebraic relation \eqref{4.3} implies that $T(r, w_1)=T(r,w_2)+O(1)$,
see~\cite[Corollary~2.1]{ishizakit:07}, which asserts that $w_1$ and $w_2$ satisfy \eqref{1.4} in place of $f_1$ and $f_2$. Remark~\ref{rem1.1} implies that \eqref{4.3}, or an identity like it, cannot be obtained in general for all pairs of distinct solutions to \eqref{1.1}, since there may exist solutions to \eqref{1.1} of distinct order of growth.
However, we are still interested in whether there exists an algebraic relation for solutions to \eqref{1.1} under some conditions. In the case of periodic coefficients, we have

\begin{theorem}\label{thm4.1}
Suppose that $A(z)$ and $B(z)$ are periodic functions of period~1 in \eqref{1.1},
and suppose that \eqref{1.1} possesses distinct transcendental meromorphic solutions $f_1(z)$ and $f_2(z)$ such that $f_j(z+2)\ne f_j(z)$, $j=1, 2$.
Then the Casoratian $f_1(z)\Delta f_2(z)-f_2(z)\Delta f_1(z)=H(z)$ is a periodic
function of period~1, and $f_1(z)$ and $f_2(z)$ satisfy the algebraic relation
\begin{multline}
A(z)((A(z)+4)f_1(z)^2f_2(z)^2-2B(z)(f_1(z)^2+f_2(z)^2))H(z)^2\\
-(A(z)B(z)(f_1(z)^2-f_2(z)^2))^2=H(z)^4,\label{4.4}
\end{multline}
where the coefficients are periodic functions of period 1.
\end{theorem}


\begin{proof}
By Proposition~\ref{prop4.1}, $f_1(z)$, $f_2(z)$ are solutions of linear homogeneous difference equation \eqref{4.1}, which can be written as
$$
f(z+2)-(A(z)+2)f(z+1)+f(z)=0.
$$
By the definition of the Casoratian,
\begin{align*}
H(z+1)&=\begin{vmatrix}
f_1(z+1)&f_2(z+1)\\
\Delta f_1(z+1)&\Delta f_2(z+1)
\end{vmatrix}
=\begin{vmatrix}
f_1(z+1)&f_2(z+1)\\
f_1(z+2)&f_2(z+2)
\end{vmatrix}  \\
&=\begin{vmatrix}
f_1(z+1)&f_2(z+1)\\
(A(z)+2)f_1(z+1)-f_1(z)&(A(z)+2)f_2(z+1)-f_2(z)
\end{vmatrix}\\
&=H(z).
\end{align*}
This shows that $H(z)$ is a periodic function of period 1. We write
\begin{equation}
f_2(z)\Delta f_1(z)=f_1(z)\Delta f_2(z)-H(z).\label{4.5}
\end{equation}
From \eqref{4.5} and \eqref{1.1},
\begin{align*}
f_2(z)^2&A(z)(f_1(z)f_1(z+1)-B(z))\\
&=f_2(z)^2(\Delta f_1(z))^2=(f_1(z)\Delta f_2(z)-H(z))^2\\
&=f_1(z)^2(\Delta f_2(z))^2-2f_1(z)H(z)\Delta f_2(z)+H(z)^2\\
&=f_1(z)^2A(z)(f_2(z)f_2(z+1)-B(z))-2f_1(z)H(z)\Delta f_2(z)+H(z)^2.
\end{align*}
Combining this and the definition of $H(z)$, we have
\begin{equation}
\Xi(z)=2f_1(z)H(z)\Delta f_2(z), \label{4.6}
 \end{equation}
where $\Xi(z)$ is a polynomial in $f_1(z)$ and $f_2(z)$ given by
 \begin{equation}
\Xi(z)=A(z)H(z)f_1(z)f_2(z)-A(z)B(z)(f_1(z)^2-f_2(z)^2)+H(z)^2,\label{4.7}
  \end{equation}
where the coefficients are periodic functions of period 1. Using \eqref{4.6} and \eqref{1.1}, we may compute
\begin{equation}
\Xi(z)^2=4f_1(z)^2H(z)^2A(z)(f_2(z)\Delta f_2(z)+f_2(z)^2-B(z)).\label{4.8}
\end{equation}
Eliminating $\Delta f_2(z)$ and $\Xi(z)$ from \eqref{4.6}, \eqref{4.7} and \eqref{4.8}, we obtain \eqref{4.4}. We have thus proved Theorem~\ref{thm4.1}.
\end{proof}

\begin{remark}
If in Theorem~\ref{thm4.1} both $f_1$ and $f_2$ are periodic of period $2$, but at least one of them is not $1$-periodic, then the algebraic relation \eqref{4.4} holds, but now with $H(z)$ having period $2$ instead of $1$. If both $f_1$ and $f_2$ are periodic functions of period $1$, then the Casoratian $H(z)$ vanishes and \eqref{4.4} does not hold.
\end{remark}


\begin{example}\label{ex4.1} Suppose that $a\ne n\pi$, $n=0, \pm1,\pm2,\dots$ in Example~\ref{ex2.1}. Then $f_1(z)=\sin az$ and $f_2(z)=\cos az$ are not periodic functions of period~2,
and the Casoratian $H(z)=-\sin a$. In this case $A(z)=-4\sin^2\frac{a}{2}$ and $B(z)=\cos^2\frac{a}{2}$.
The relation \eqref{4.4} now reduces to
    $$-\sin^4 a\cdot(f_1(z)^2+f_2(z)^2-2)(f_1(z)^2+f_2(z)^2)=\sin^4 a,$$
i.e., $f_1(z)^2+f_2(z)^2=1$.
\end{example}

\section{Value distribution and Growth relation of solutions}\label{VG}
We suppose that \eqref{1.1} possesses two distinct meromorphic solutions $f_1(z)$ and $f_2(z)$. We discuss the relation between $f_1(z)$ and $f_2(z)$, for instance, whether \eqref{1.4} holds or not.
To this end, we define a set $\mathcal{M_{\text{E}}}$ of meromorphic functions as follows.
If $f\in \mathcal{M_{\text{E}}}$ then the multiplicities of all zeros and all of poles of $f(z)$ are even. In case $f(z)$ has no zeros and no poles, we allow $f$ to belong to $\mathcal{M_{\text{E}}}$. Denote by $G(f)$ a quadratic polynomial in $f$
\begin{equation}
G(f)=(A(z)+4)f^2-4B(z),\label{5.1}
 \end{equation}
where $A(z)$ and $B(z)$ are meromorphic functions given in \eqref{1.1}.
For example, the function $f_b(z)$ in Example~\ref{ex2.2} solves the difference equation \eqref{2.3} with $A(z)=-4$ and $B(z)=1$ and satisfies
$G(f_b)=-4$. This means that $G(f_b)\in \mathcal{M_{\text{E}}}$.
The function $f_\beta(z)$ in Example~\ref{ex2.3} solves the difference equation \eqref{2.5} with $A(z)=-4\beta(z)^2/(\beta(z)^2-4)$ and $B(z)=1$. We have
$$
G(f_\beta)=-4\left(\frac{(e^{2\pi i z}+1)\beta(z)-4e^{\pi i z}}{e^{2\pi i z}-1}\right)^2\cdot \frac{1}{\beta(z)^2-4},
$$
which implies that $G(f_\beta)$ does not belong to $\mathcal{M_{\text{E}}}$ for some $\beta(z)$.
Let $n_{\text{odd}}(r,f)$ be a counting function which counts poles in $|z|<r$ whose multiplicities are
odd and let $\overline n_{\text{odd}}(r,f)$ count odd multiple poles once for each occurrences.
The integrated counting functions $N_{\text{odd}}(r,f)$ and $\overline N_{\text{odd}}(r,f)$ are
defined in a usual manner. We set
    $$N_{\text{O}}(r,f)=N_{\text{odd}}(r,f)+N_{\text{odd}}(r,1/f)$$
and
    $$\overline N_{\text{O}}(r,f)=\overline N_{\text{odd}}(r,f)+\overline N_{\text{odd}}(r,1/f)$$
for convenience. Let $f(z)$ be a transcendental meromorphic solution to \eqref{1.1}. We call $f(z)$ an admissible solution if $f(z)$ satisfies $T(r,A)=S(r,f)$ and $T(r,B)=S(r,f)$.

\begin{theorem}\label{thm5.1}
Assume that $A(z)$ and $B(z)$ are non-constant meromorphic functions in \eqref{1.1}.
Suppose \eqref{1.1} possesses two distinct admissible solutions $f_1(z)$ and $f_2(z)$. Then \eqref{1.4} holds or
\begin{equation}
\overline N_{\text{O}}(r,G(f_1))=\overline N_{\text{O}}(r,G(f_2)).\label{5.2}
\end{equation}
\end{theorem}


\begin{proof}
For the sake of simplicity, we write $f_1(z)=a(z)$ and $f_2(z)=f(z)$. It is possible that $f(z)=-a(z)$, in which case \eqref{1.4} clearly holds.
We assume that $f(z)\ne-a(z)$ below.
Since $f(z)$ and $a(z)$ are admissible solutions, neither $f(z)$ nor $a(z)$ is periodic function of period 1, i.e., $\Delta f(z)\not\equiv0$ and $\Delta a(z)\not\equiv0$, by \eqref{1.1}.
We set
\begin{equation}
g(z)=\frac{f(z)-a(z)}{f(z)+a(z)}\quad\text{or}\quad f(z)=-a(z)\frac{g(z)+1}{g(z)-1}\label{5.3}
\end{equation}
in \eqref{1.1}. We note that $g(z)-1=-2a(z)/(a(z)+f(z))\not\equiv0$ and $g(z)+1=2f(z)/(a(z)+f(z))\not\equiv0$.
If $g(z)$ reduces to a small function with respect to $f_1$ and $f_2$, by \eqref{5.3} and the first main theorem due to Nevanlinna, we obtain \eqref{1.4}.
Below we assume that $g(z)$ is a transcendental meromorphic function and not a small function with respect to $f(z)$ and $a(z)$.

By combining \eqref{5.3} and \eqref{1.1} we eliminate $f(z)$, and then we use \eqref{1.1} again to eliminate $\Delta a(z)$. Then we obtain
\begin{equation}
C_2(z)(\Delta g(z))^2+C_1(z)\Delta g(z)+C_0(z)=0\label{5.4}
 \end{equation}
with
\begin{align}
C_0(z)=&4A(z)B(z)(g(z)-1)^2g(z),\label{5.5}\\
C_1(z)=&2(g(z)-1)\big(2A(z)B(z)(g(z)-1)+a(z)^2A(z)(g(z)+1)\nonumber\\
&\hspace{2cm}+\Delta a(z)(A(z)+2)a(z)(g(z)+1)\big),\label{5.6}\\
C_2(z)=&2a(z)\big(a(z)A(z)+\Delta a(z)(A(z)+2)\big)g(z)\nonumber\\
&-2a(z)\big(2a(z)+a(z)A(z)+\Delta a(z)(A(z)+2)\big).\label{5.7}
\end{align}
We consider first the case $C_2(z)\equiv0$. Since $a(z)\not\equiv0$, we have
\begin{equation}
 \Delta a(z)(A(z)+2)(g(z)-1)=a(z)(A(z)(1-g(z))+2). \label{5.8}
  \end{equation}
We note that $\Delta a(z)\not\equiv0$ and $g(z)-1\not\equiv0$ as mentioned above.
By our assumption, we have $A(z)+2\not\equiv0$.
Thus by solving $a(z+1)$ from \eqref{5.8}, and then using \eqref{1.1} and \eqref{5.8} to eliminate $\Delta a(z)$, we obtain
\begin{multline}
\big(A(z)(A(z)+4)g(z)^2-(A(z)+2)^2\big)a(z)^2\\
=A(z)B(z)(A(z)+2)^2(g(z)-1)^2,\label{5.9}
 \end{multline}
which implies that $a(z)^2$ is represented by a rational function in $g(z)$ of degree 2
with rational coefficients. By means of the Valiron--Mohon'ko theorem~\cite{mohonko:71},~\cite[Theorem~2.2.5]{laine:93} and \eqref{5.8}, we have $T(r,a)=T(r,g)+S(r,f)$.
It follows from \eqref{1.1} and \eqref{5.9},
\begin{align}
f(z)^2&=a(z)^2\left(\frac{g(z)+1}{g(z)-1}\right)^2\nonumber\\
&=\frac{A(z)B(z)(A(z)+2)^2(g(z)+1)^2}{A(z)(A(z)+4)g(z)^2-(A(z)+2)^2}.\label{5.10}
\end{align}
We apply the Valiron--Mohon'ko theorem to \eqref{5.10} again,
and obtain $T(r,f)=T(r,g)+S(r,f)$. Hence we have $T(r,a)=T(r,f)+S(r,f)$, which implies that \eqref{1.4} holds.

We assume now that  $C_2(z)\not\equiv0$.
We write by \eqref{5.4},
\begin{equation}
\left(\Delta g(z)+\frac{C_1(z)}{2C_2(z)}\right)^2=\frac{C_1(z)^2-4C_0(z)C_2(z)}{4C_2(z)^2}.\label{5.11}
\end{equation}
Using that $a(z)$ satisfies \eqref{1.1} again and recalling \eqref{5.3}, we obtain
\begin{equation}
C_1(z)^2-4C_0(z)C_2(z)=\frac{64a(z)^4a(z+1)^2A(z)G(f)}{(f(z)+a(z))^4},\label{5.12}
\end{equation}
where $G(f)$ is defined in \eqref{5.1}.
By \eqref{5.11} and \eqref{5.12}, we see that $AG(f)\in \mathcal{M_{\text{E}}}$.
We change the roles of $f(z)$ and $a(z)$ in \eqref{5.3}, and apply the same arguments above to obtain $AG(a)\in \mathcal{M_{\text{E}}}$. Let $z_0$ be a zero or a pole of $A(z)$
of odd multiplicity. Then we see that both of $G(f)$ and $G(a)$ have a zero or a pole
at $z_0$ of odd multiplicity, which shows that \eqref{5.2} holds.
\end{proof}

The identity \eqref{5.2} for two admissible solutions of \eqref{1.1} gives an exact invariant quantity in terms of Nevanlinna's functions, which corresponds to a counterpart of \eqref{1.4} in the differential case. In general, two admissible solutions of \eqref{1.1} do not always satisfy \eqref{1.4}. We give such an example below, which is a generalization of Example~\ref{ex2.4}.

\begin{example}\label{ex5.1} Let $Q(z)$ be a periodic function of period 1, and $h(z)$ be an
arbitrary meromorphic function.
The function
\begin{equation}
f(z)=\frac{h(z)^2+Q(z)^2}{2h(z)Q(z)}\label{5.13}
\end{equation}
satisfies a difference equation
\begin{multline}
(\Delta f(z))^2\\
=\frac{(h(z+1)-h(z))^2}{h(z)h(z+1)}\left(f(z)f(z+1)-\frac{(h(z+1)+h(z))^2}{4h(z)h(z+1)}\right).\label{5.14}
\end{multline}
We write $A(z)=(h(z+1)-h(z))^2/(h(z)h(z+1))$ and $B(z)=(h(z+1)+h(z))^2/(4h(z)h(z+1))$ for simplicity. Let $Q_1(z)$ and $Q_2(z)$ be periodic functions of period 1 such that
$T(r,h)=S(r,Q_j)$ and $T(r,h(z+1))=S(r,Q_j)$ , $j=1, 2$. Then $f_j(z)=(h(z)^2+Q_j(z)^2)/2h(z)Q_j(z)$, $j=1, 2$ are admissible solutions to \eqref{5.14}. If we choose $Q_1(z)$ and $Q_2(z)$ satisfying $T(r,Q_2)=S(r,Q_1)$, then \eqref{1.4} does not hold. In fact, for any solution $f(z)$ given by \eqref{5.13},
\begin{align*}
G(f)&=(A+4)f(z)^2-4B(z)\\
&=\left(\frac{(h(z+1)-h(z))^2}{h(z)h(z+1)}+4\right)\left(\frac{h(z)^2+Q(z)^2}{2h(z)Q(z)}\right)^2-4\left(\frac{(h(z+1)+h(z))^2}{4h(z)h(z+1)}\right)\\
&=\frac{(h(z)+h(z+1))^2(h(z)-Q(z))^2(h(z)+Q(z))^2}{4h(z)^3h(z+1)Q(z)^2}.
\end{align*}
Hence we have \eqref{5.2} as
\begin{equation*}
\overline N_{\text{O}}(r,G(f_1))=\overline N_{\text{O}}(r,G(f_2))=\overline N_{\text{O}}(r,h(z)^3h(z+1)).
\end{equation*}
\end{example}


\def\cprime{$'$}
\providecommand{\bysame}{\leavevmode\hbox to3em{\hrulefill}\thinspace}
\providecommand{\MR}{\relax\ifhmode\unskip\space\fi MR }
\providecommand{\MRhref}[2]{%
  \href{http://www.ams.org/mathscinet-getitem?mr=#1}{#2}
}
\providecommand{\href}[2]{#2}

\vspace{0.5cm}

\noindent
Katsuya Ishizaki\\
The Open University of Japan\\
2-11 Wakaba, Mihama-ku, Chiba\\
261- 8586 JAPAN\\
E-mail address: ishizaki@ouj.ac.jp
\vspace{0.5cm}

\noindent
Risto Korhonen\\
Department of Physics and Mathematics\\
University of Eastern Finland\\
Joensuu Campus, P. O. Box 111\\
FI-80101 Joensuu, Finland\\
E-mail address: risto.korhonen@uef.fi

\end{document}